\documentclass[reqno,11pt, heading]{article}

\usepackage[utf8x]{inputenc}
\usepackage[english,russian]{babel}
\usepackage{graphicx}
\usepackage{amsmath,amsthm,amsfonts,amssymb,amscd}
\usepackage{float}
\usepackage{caption}
\usepackage{xcolor}
\usepackage{authblk}

\theoremstyle{plain}
\newtheorem{thm}{Theorem}

\newtheorem{lm}[thm]{Lemma}
\newtheorem{rem}[thm]{Remark}

\newtheorem*{dfn*}{Definition}

\newcommand{\Z}{\mathbf{Z}}

\usepackage[figurename=Figure]{caption}

\begin{document}
\title{Two-dimensional self-interlocking structures in three-space}

\author[1]{V.O. Manturov\footnote{
The work of V.O.Manturov was funded by the development program of the Regional Scientific and Educational Mathematical Center of the Volga Federal District, agreement N 075-02-2020.
}}
\author[2]{A.Ya. Kanel-Belov\footnote{Alexei Kanel-Belov was supported by Russian Science Foundation grant No. 17-11-01377.}}
\author[3]{S. Kim}
\affil[1,2]{Moscow institute of physics and technology (Russia)}
\affil[1]{Kazan Federal University (Russia)}
\affil[1]{Northeastern University (China)}
\affil[2]{Bar-Ilan University (Israel)}
\affil[3]{Jilin university (China)}

\date{}

\maketitle

\section{Introduction. Statement of the problem}

It is well known that if there is a finite set of convex bodies on the plane whose interiors do not overlap, then there is at least one {\em extreme} one among these bodies --- one that can be continuously moved ``to infinity'' (outside the large ball containing other bodies), leaving all other bodies motionless. Moreover, if all these bodies are balls, then in a space of any dimension one can find a body that is carried away ``to infinity'' (see \cite {Turgor}).

It was observed that in a space of dimension three and higher, the phenomenon of {\em self-interlocking structures} takes place. {\em Self-interlocking structure} is a collection of convex bodies with non-overlapping interiors such that if one fix everything except any one, the rest cannot be ``carried away to infinity.'' This property is equivalent to the following: {\it any infinitely small motion is possible only as part of the joint motion of all bodies together (as a single rigid body).} For the history of the discovery of self-interlocking structures, as well as their examples and applications, see \cite{DEBP1, DEBP2, DEPKB}.

This topic becomes popular both in pure mathematics and in applications to both architecture and natural sciences (see \cite{Belov_q}). A number of articles are devoted to it (for example, \cite{DEBP1, DEBP2, DEPKB, ActaAstronautica, Suttle, PhysReview}), both in popular (\cite{Belov_q}) and in top-rated journals, including ``Nature'' \cite{Nature}. There are a number of patents. A similar idea is already used when creating body armor \cite{bodyarmor}. As a result of the megagrant won by Yu. Estrin arose (and successfully exists) the laboratory \cite{lab-Estrin}.

Various research groups have sprung up at different times, for example, Thomas Siegmund's group Thomas Siegmund \cite{team-ThomasSiegmund} (see publications of this group, for example, \cite{KhandelwalS.SiegmundT.CipraR.J.andBoltonJ.S., KhandelwalS.SiegmundT.CipraR.J.andBoltonJ.S.1, FengY.SiegmundT.HabtourE.RiddickJ., FengY.SiegmundT.HabtourE.RiddickJ.1, MatherACipraR.JSiegmundT.}.

We will provide links to some other groups and individual researchers around the world Francois Barthelat \cite{FrancoisBarthelat}, Yves Brechet \cite{YvesBrechet}, Andrey Molotnikov \cite{AndreyMolotnikov},
Giuseppe Fallacara \cite{GiuseppeFallacara}, Vera Viana \cite{VeraViana}, and also \cite{dartmouth}.
We also indicate some of the work of the group and the Netherlands \cite{heronjournal} and in Technione \cite{grobman1, grobman2, sagepub}

The available structures are based on the consideration of layers of cubes, tetrahedra and octahedra and their variations (see Fig. \ref{stacking_cubes}, \ref{stacking_tetrahedra}, \ref{stacking_octa}).

\begin{figure}[!htb]
 \begin{minipage}{.5\textwidth}
 \centering\includegraphics[width=\linewidth]{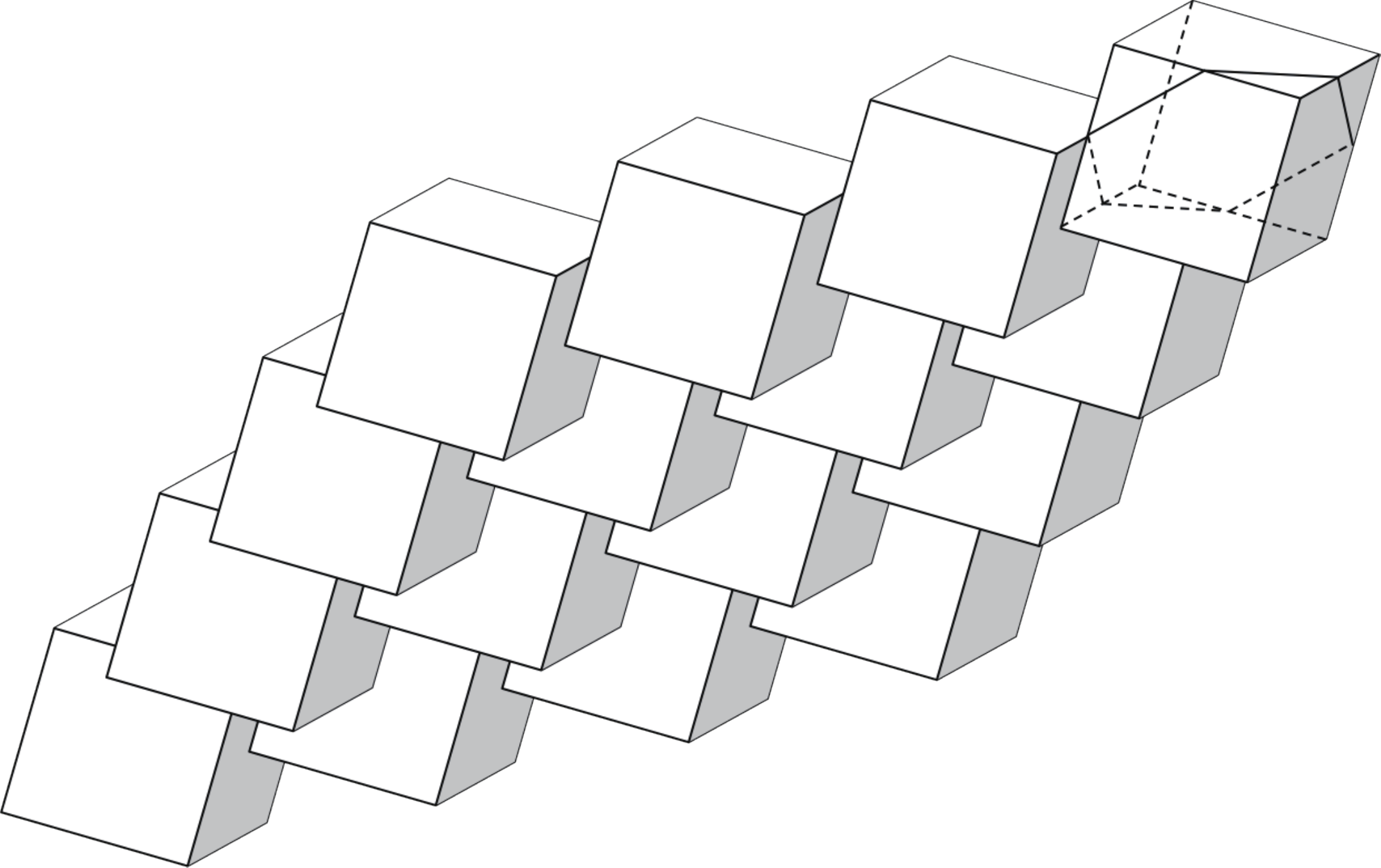}
\caption{Stacking of cubes}
\label{stacking_cubes}
  \end{minipage}\hfill
   \begin{minipage}{.5\textwidth}
 \centering\includegraphics[width=\linewidth]{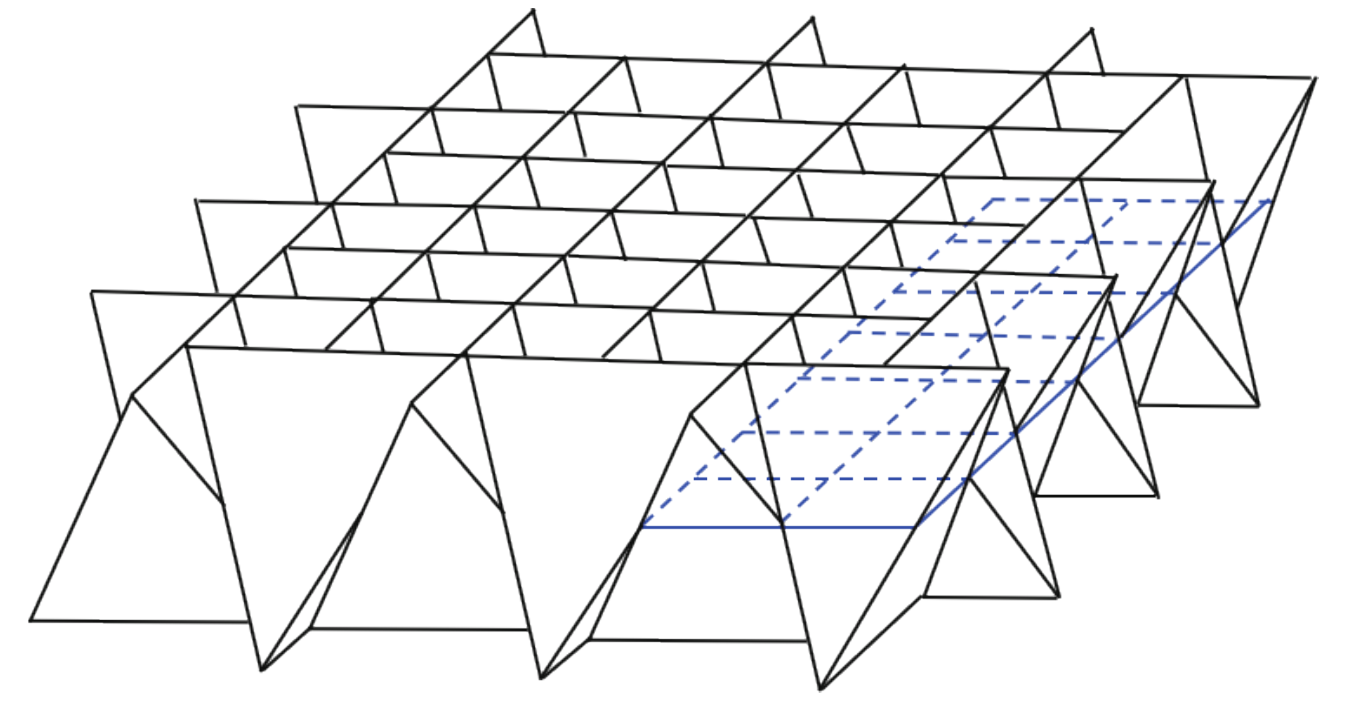}
\caption{Stacking of tetrahedra}
\label{stacking_tetrahedra} 
  \end{minipage}
  \vspace{5mm}

  \centering\includegraphics[width=150pt]{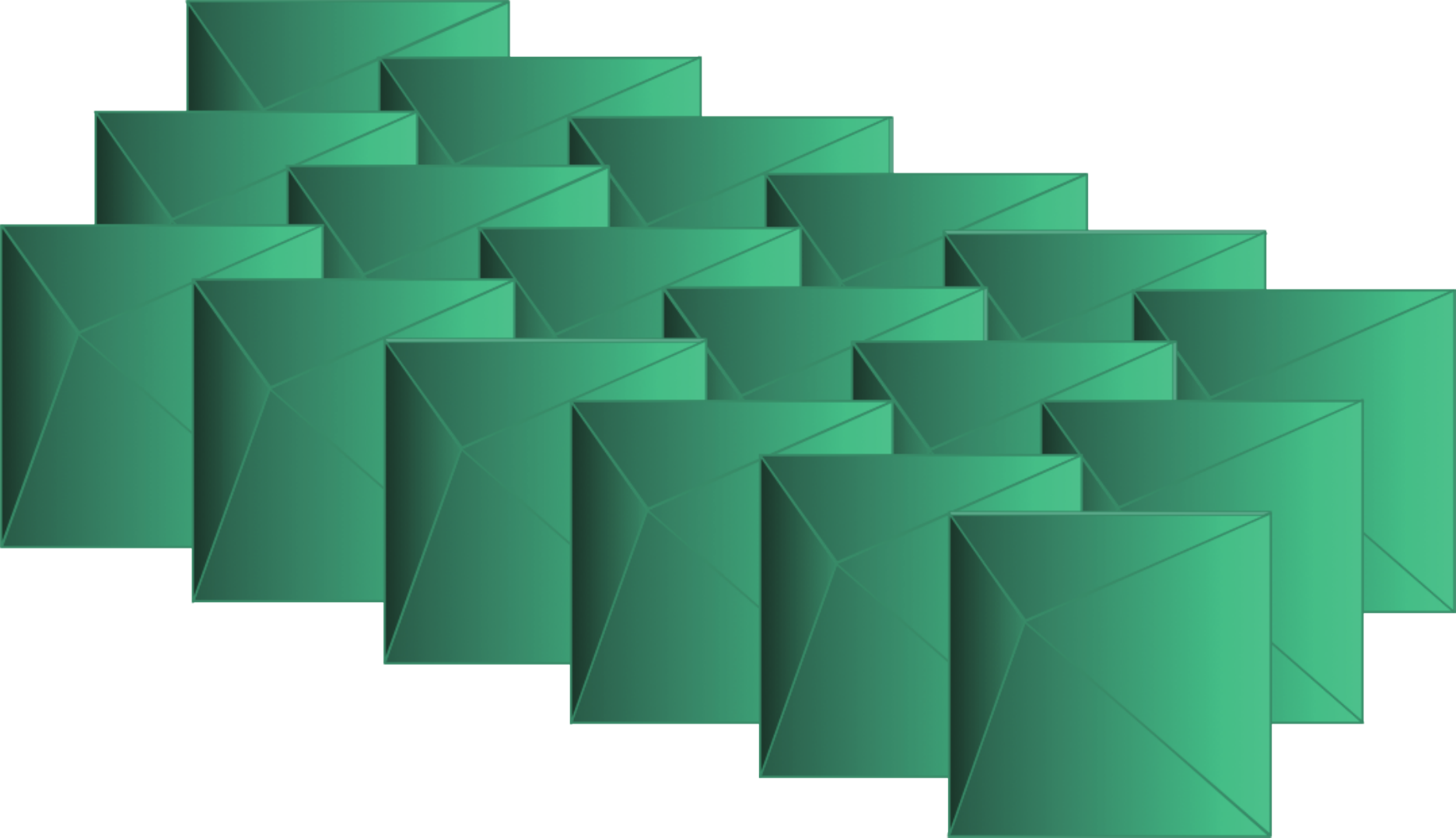}
\caption{Stacking of octahedra}
\label{stacking_octa}
\end{figure}

These structures have a feature: when the border is fixed along the perimeter (see Fig. \ref{model}), the structure becomes rigid and does not collapse.

\begin{figure}[!htb]
 \begin{minipage}{.4\textwidth}
 \centering\includegraphics[width=\linewidth]{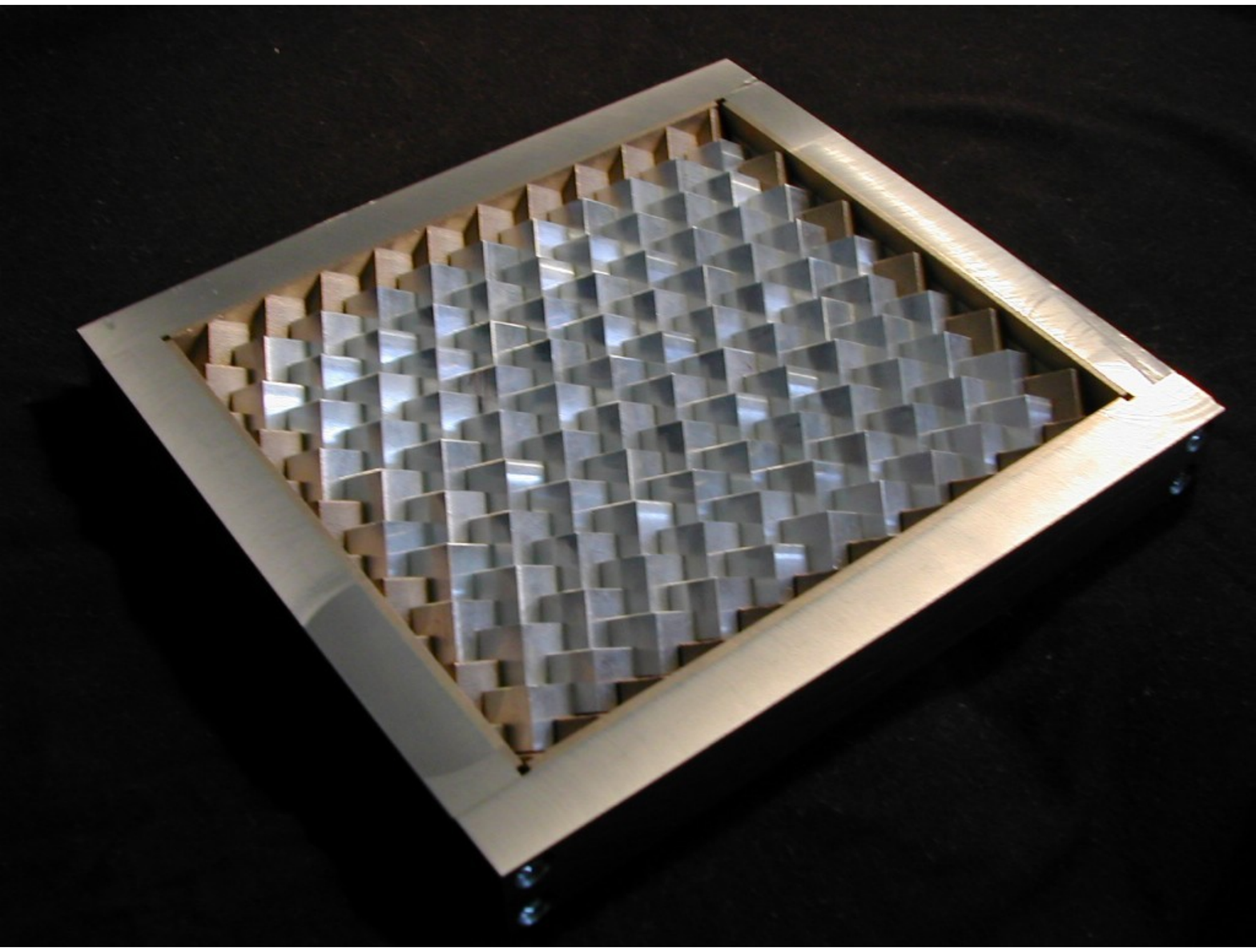}
  \end{minipage}\hfill
   \begin{minipage}{.4\textwidth}
 \centering\includegraphics[width=\linewidth]{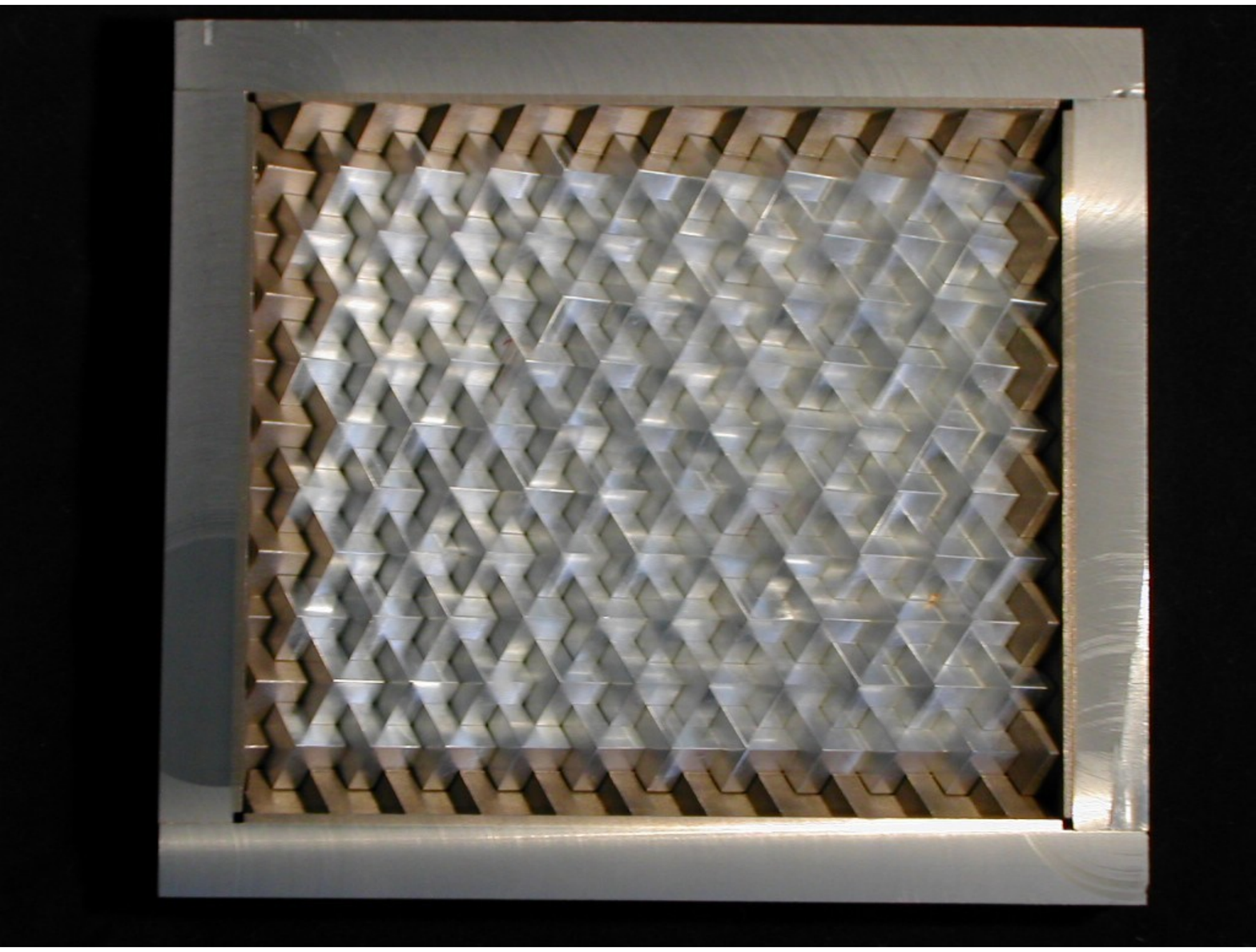}
  \end{minipage}

  \vspace{-3cm}
  \centering\includegraphics[width=140pt]{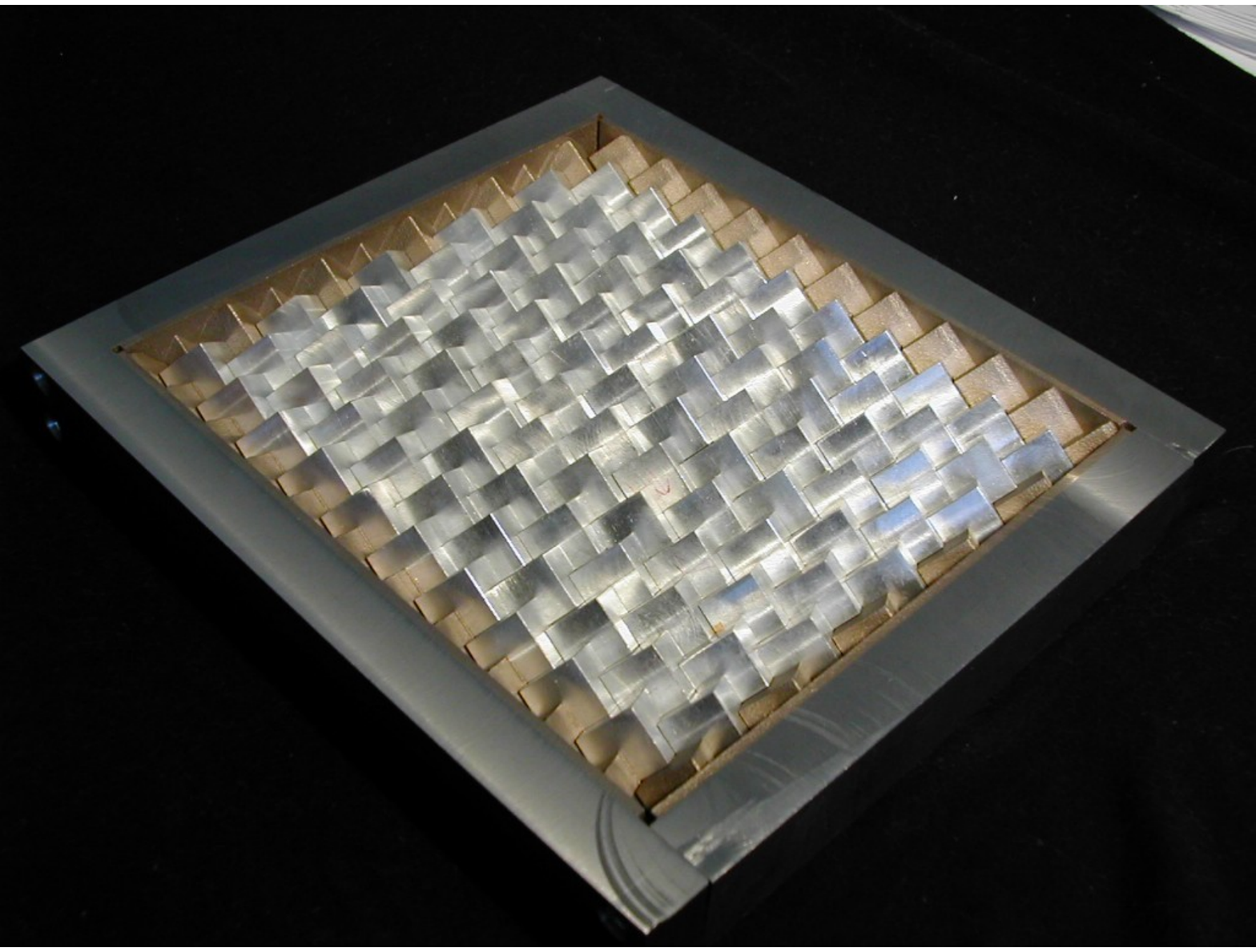}
    \vspace{-1cm}
    \caption{Model}\label{model}
  \end{figure}

In addition to quasi-flat structures, there are structures where an interlocking occurs in several layers simultaneously (see Fig. \ref{stuck}). These are questions related to the {\it cladding} of a flat layer with a fixed perimeter.

\begin{figure}[!htb]

  \centering\includegraphics[width=200pt]{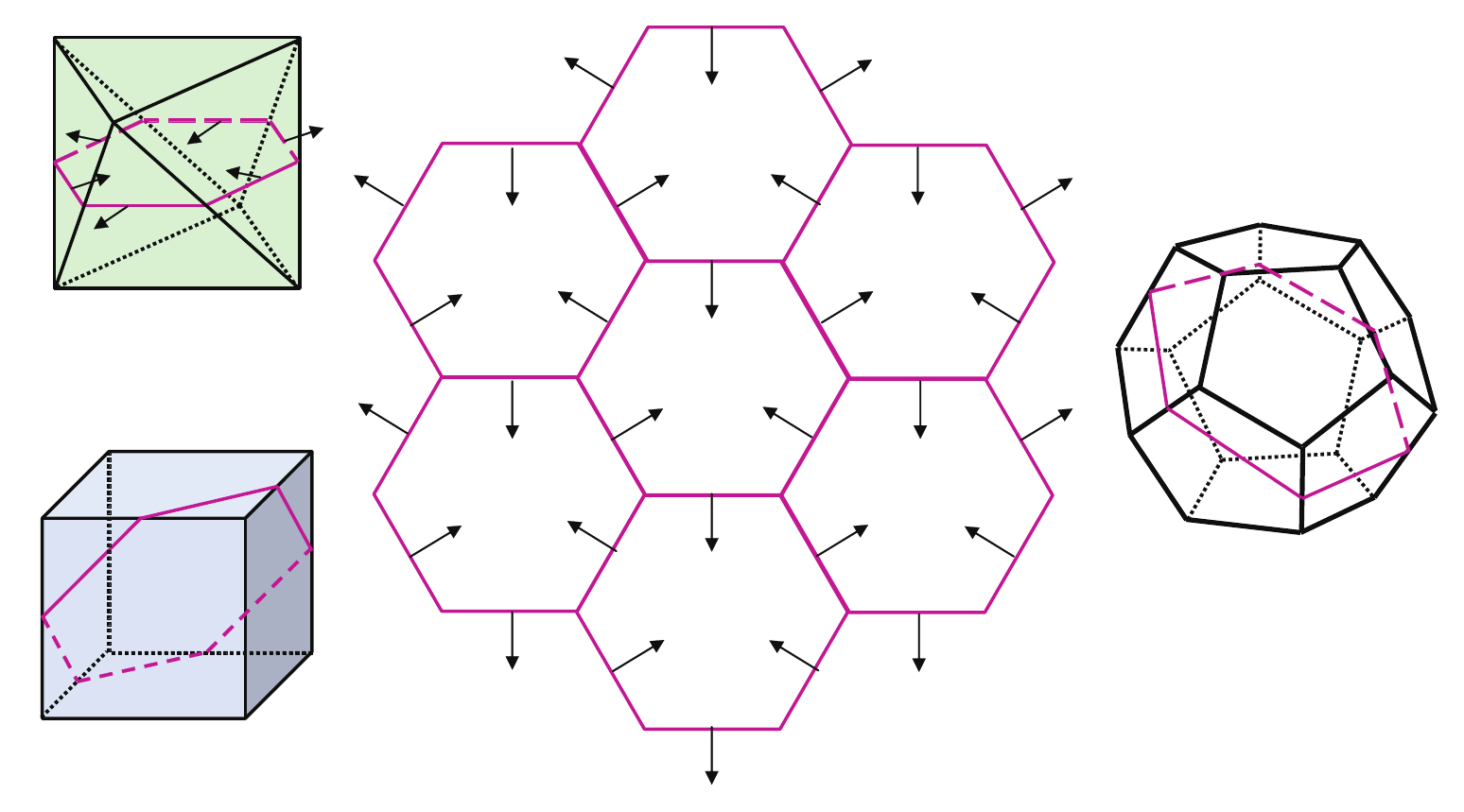}
    \caption{Interlocking in layers}\label{stuck}
  \end{figure}
 
 In this paper, we consider principally different structures. Self-interlocked {\it decahedra} are assembled from thin tiles, from which, in turn, second order structures are assembled. In particular, the construction of a {\bf column} composed of decahedra, which is stable when two extremes, but not the entire boundary of the layer, are fixed, as in the structures studied earlier. In addition, we present a structure composed of flat tiles (such an arrangement is not possible on the plane).
Apparently, this paper is the first of this kind both in terms of the fact that it is enough to hold the two extreme objects of the column, and in terms of working with two-dimensional elements.

These constructions are interesting in the following point of view. Previously created self-interlocking structures were {\bf rigid}, at the same time be interesting structures in a certain (controlled) sense {\bf flexible}, whose elements can move within a certain framework. This can be interesting both in an architectural point of view, and when controlling the process of relaxation and damping of various kinds of waves, as well as other manipulations.


This paper is organized as follows. In Section \ref{Decaedr} a decahedron obtained from a dodecahedron by removing two opposite faces and a small ``pull'' of some faces is constructed. Then, in Section \ref{Structure}, we present the self-interlocking structure in the form of
a ``tower'' consisting of a set of ``sequentially nested'' decahedra, each of which is ``blocked inward'' by a decagon and blocked on the outside by adjacent side faces.
In Section \ref{Structure1} we discuss decahedron necklaces and pillars.
Finally, Section \ref{Openproblems} is devoted to some open problems.

\section{Construction and coordinates of decahedra}\label{Decaedr}
Consider the faces of a dodecahedron on the plane (obtained, for example, by projecting all the faces inside one face). The rules for  ``stretching edges'' are shown in Fig. \ref{long}.

 \begin{figure}
\centering\includegraphics[width=200pt]{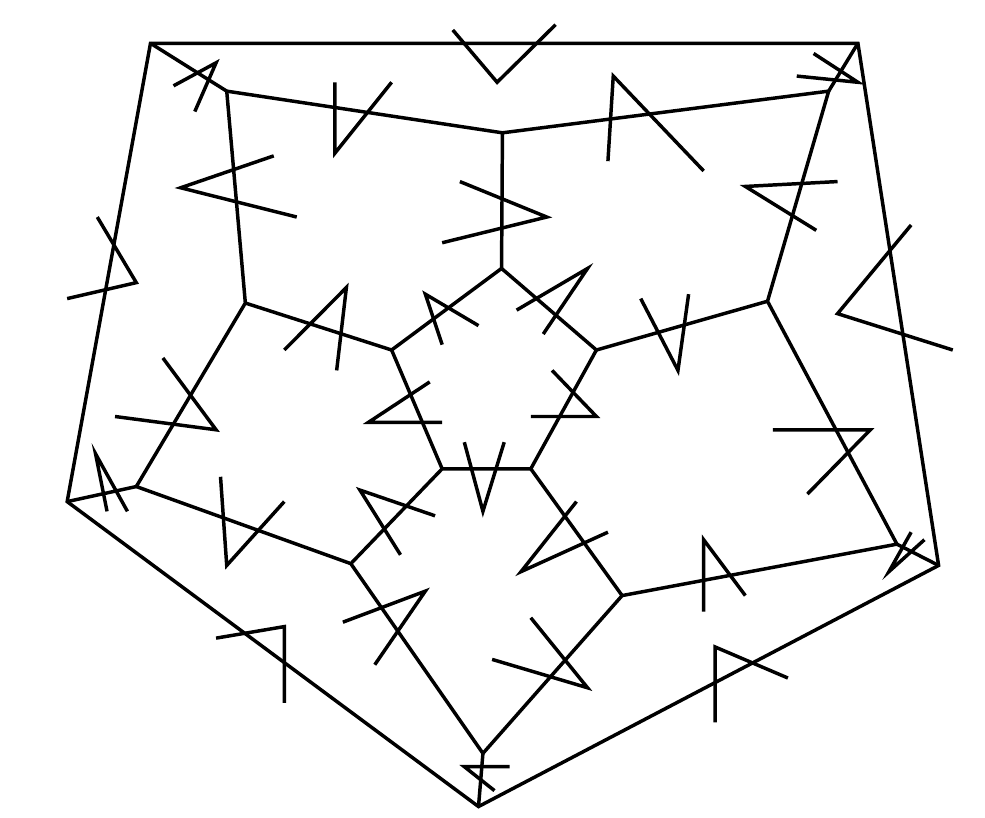}
\caption{ Faces of a dodecahedron, which are blocked by each others}
\label{long}
\end{figure}

In Fig. \ref{long} twelve faces of the dodecahedron $M$ are depicted 
in such a way that for each pair of two edge-adjacent faces one face is marked as ``bigger'' than another ($a> b$ means that ``the face $ b $ is blocked by the face $a$''). 
From the construction, the following lemma can be obtained.

\begin{lm}
Each face of the decahedron $M$ is blocked by three of the five faces adjacent to it.
\end{lm}

Coordinates of the vertices of the dodecahedron $M$ in cylindrical coordinates
are written in Fig. \ref{dodecahedron}. In this paper, the bottom 5 pentagons, which are colored by blue in Fig. \ref{dodecahedron}, are called {\em the bottom belt}, and the top 5 pentagons, which are colored by red in Fig. \ref{dodecahedron}, are called {\em the top belt}.

The faces with wings of the decahedron ${\widetilde M}$ are obtained from the faces of the dodecahedron $M$ as follows. For each pentagonal face $p$ of dodecahedron $M$ centered at the point $O_ {p}$ a face $p'$ with wings is obtained from the face $p$ by attaching wings to edges along which the face $p$ is bigger than the adjacent face, see Fig.~\ref {faces} and~\ref{wing_decahedron}.

 \begin{figure}
\centering\includegraphics[width=150pt]{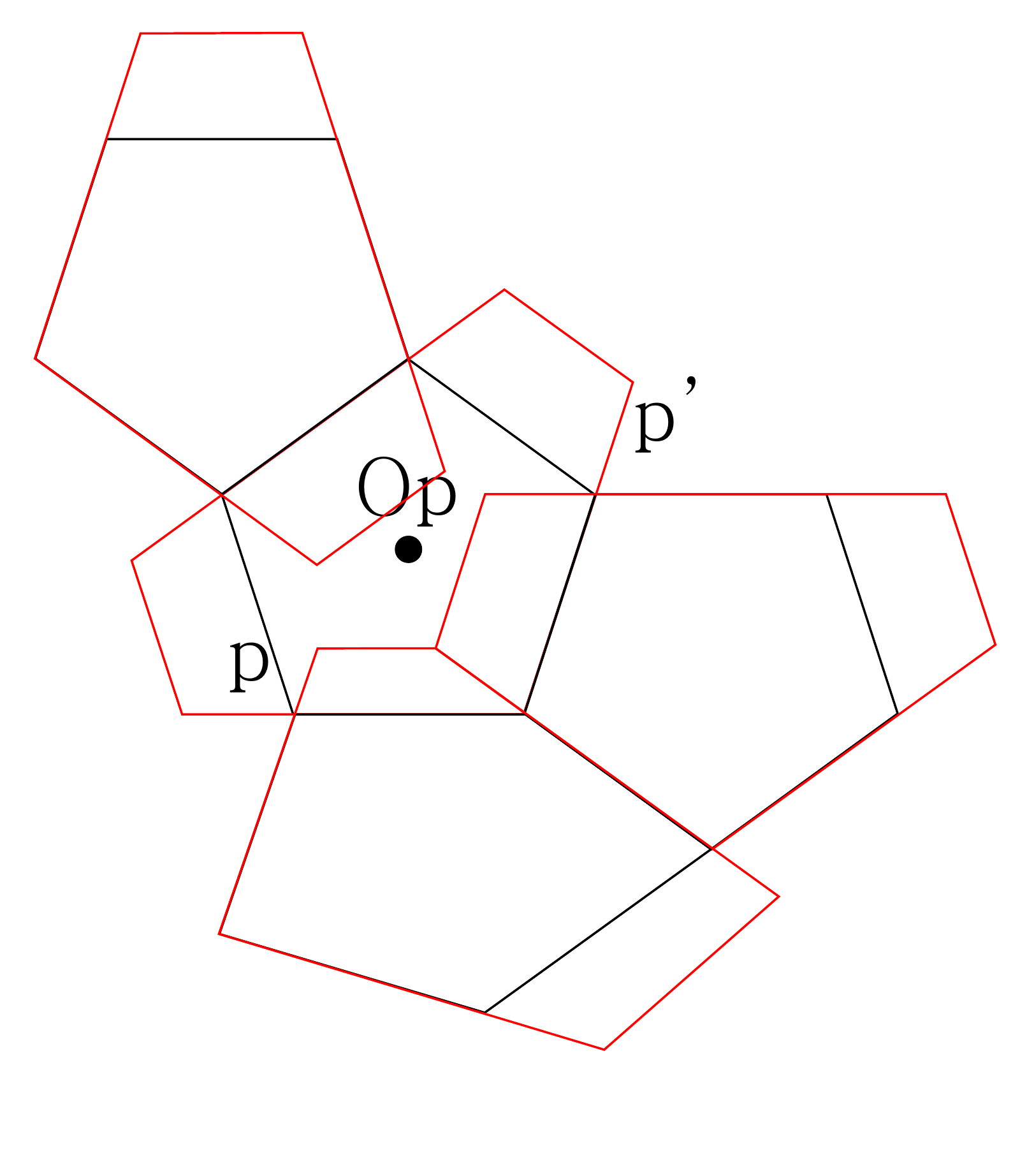}
\caption{Faces with wings of the decahedron ${\widetilde M}$}
\label{faces}
\end{figure}

 \begin{figure}
\centering\includegraphics[width=300pt]{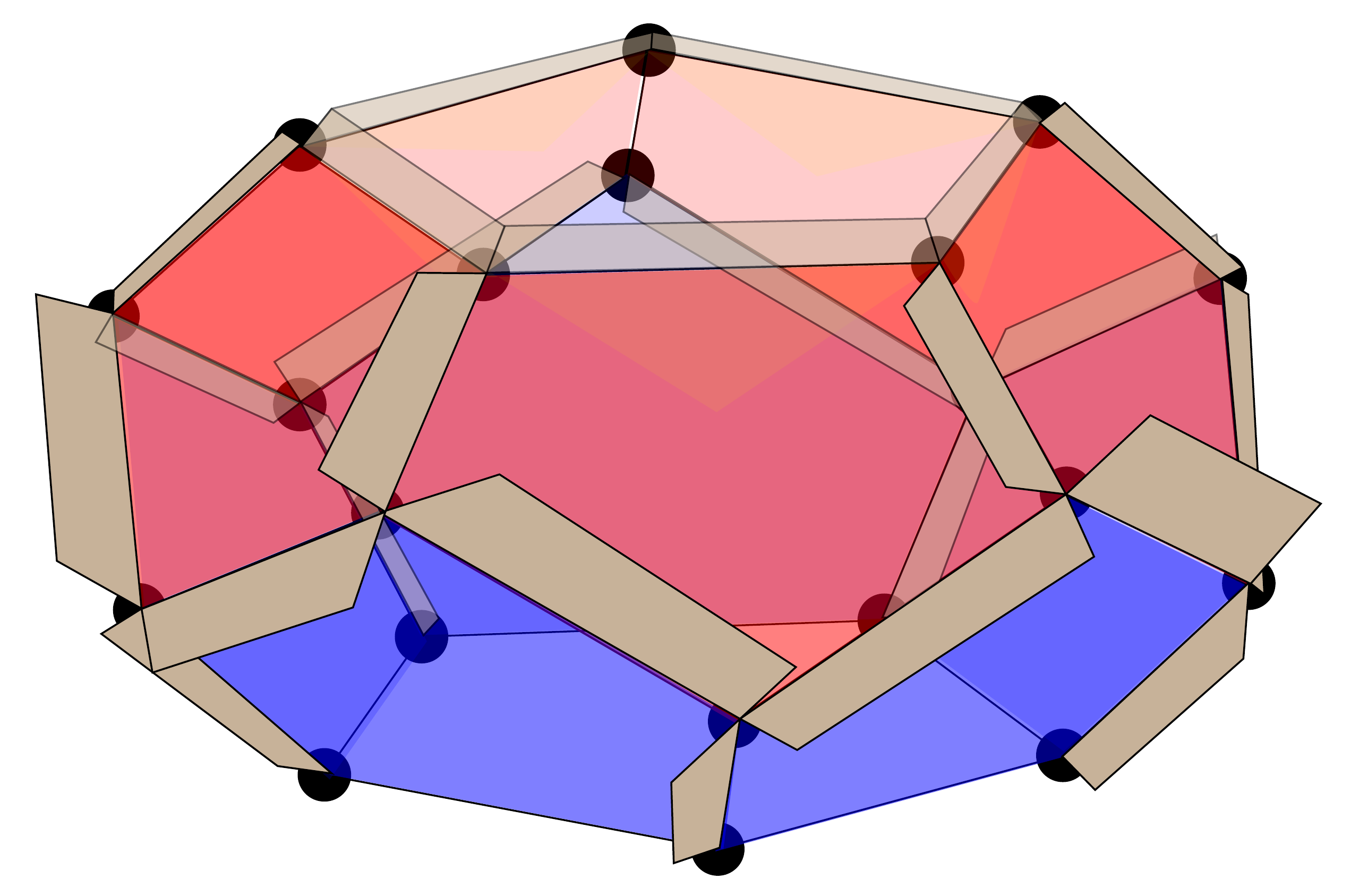}
\caption{Faces with wings of the decahedron ${\widetilde M}$ in $\mathbb{R}^{3}$}
\label{wing_decahedron}
\end{figure}

We say that the face $\Pi$ with wings of the decahedron ${\widetilde M}$ {\em cannot move in the direction of the vector $ \overrightarrow{v}$} if it is impossible to move the face
$\Pi$ parallel along the vector $ \overrightarrow{v}$ in $cl(\mathbb{R}^{3} \backslash ({\widetilde M} - \Pi))$. For each face, the direction of the vector $ \overrightarrow{w} $ such that $ \overrightarrow{w} \cdot \overrightarrow{v}> 0 $ or $ \overrightarrow{w} \cdot \overrightarrow{u}> 0$ is called {\em outward}, where $ \overrightarrow{v} $ and $\overrightarrow{u}$ are vectors perpendicular to one of faces of the lower belt and the upper belt as described in Fig. \ref{faces_R3}. If $ \overrightarrow{w} \cdot \overrightarrow{v} <0 $ or $ \overrightarrow{w} \cdot \overrightarrow{u} <0 $, then the direction of the vector $ \overrightarrow{w} $ is called {\em inward}. If $\overrightarrow{w} \cdot \overrightarrow{v} =0$ or $ \overrightarrow{w} \cdot \overrightarrow{u} =0 $, then we call it {\em a vector on the face}.

 \begin{figure}
\centering\includegraphics[width=200pt]{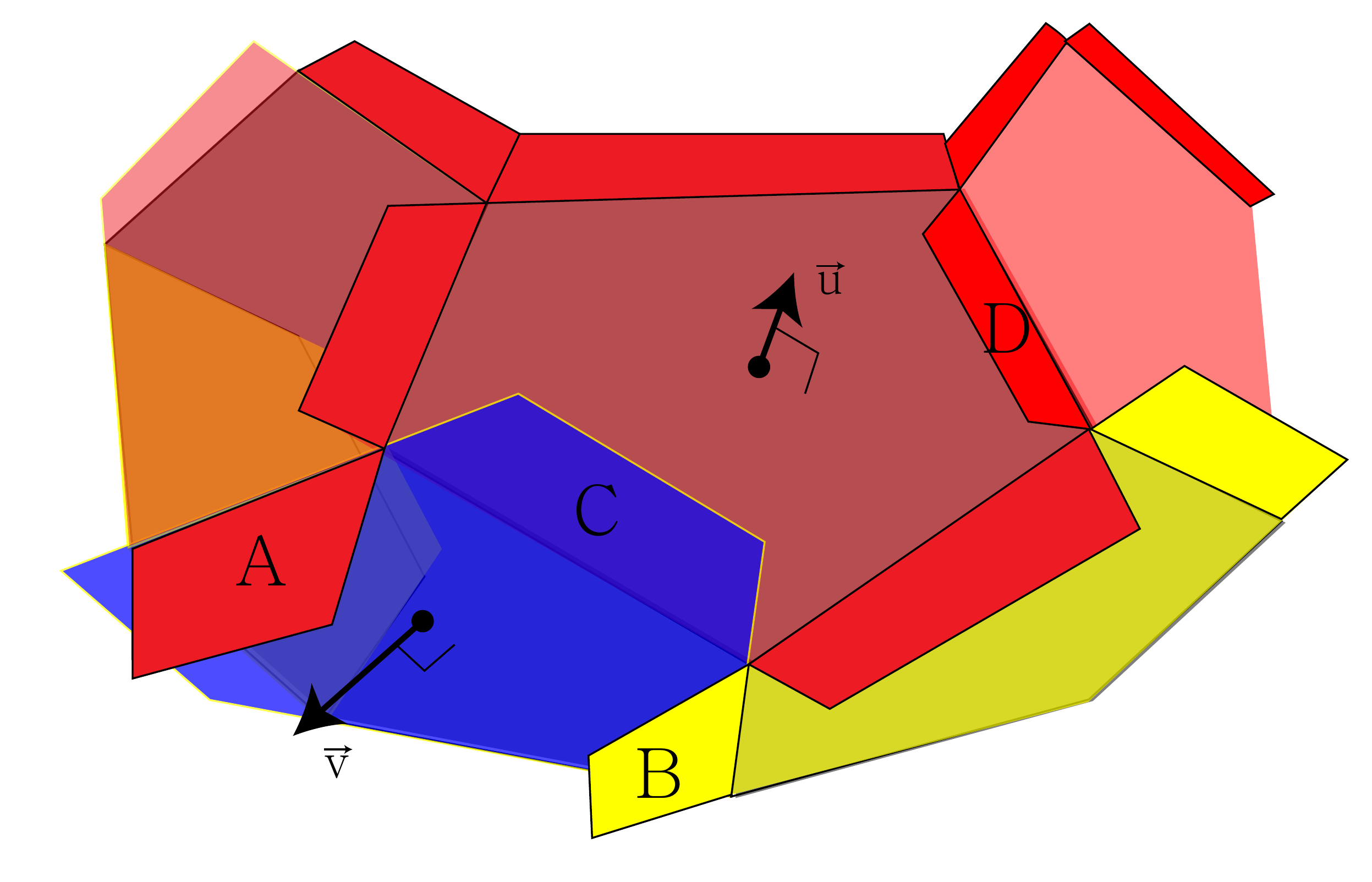}
\caption{Faces with wings of the decahedron ${\widetilde M}$ are blocked when it is translated along the perpendicular vector to the faces}
\label{faces_R3}
\end{figure}

\begin{lm}\label{lem:out}
Each face with wings of the decahedron $ {\widetilde M}$ cannot move outward if other faces are fixed.
\end{lm}

\begin{proof}[Proof]
Consider a face with wings of the lower belt (blue face) of the decahedron ${\widetilde M}$ in Fig.~\ref{faces_R3}. One can see that if the blue face with wings moves in the direction of the vector $\overrightarrow{v} $, then it is blocked by the wings $A$ and $B$ in Fig.~\ref{faces_R3}. From this observation it follows that the face of the lower belt is blocked in the direction of the vector outward. Similarly, one can show that the face with wings of the upper belt (red face) of the decahedron $ {\widetilde M} $ is blocked by two adjacent wings $ C $ and $ D $, and the proof is complete.
\end{proof}




 \begin{figure}
\centering\includegraphics[width=300pt]{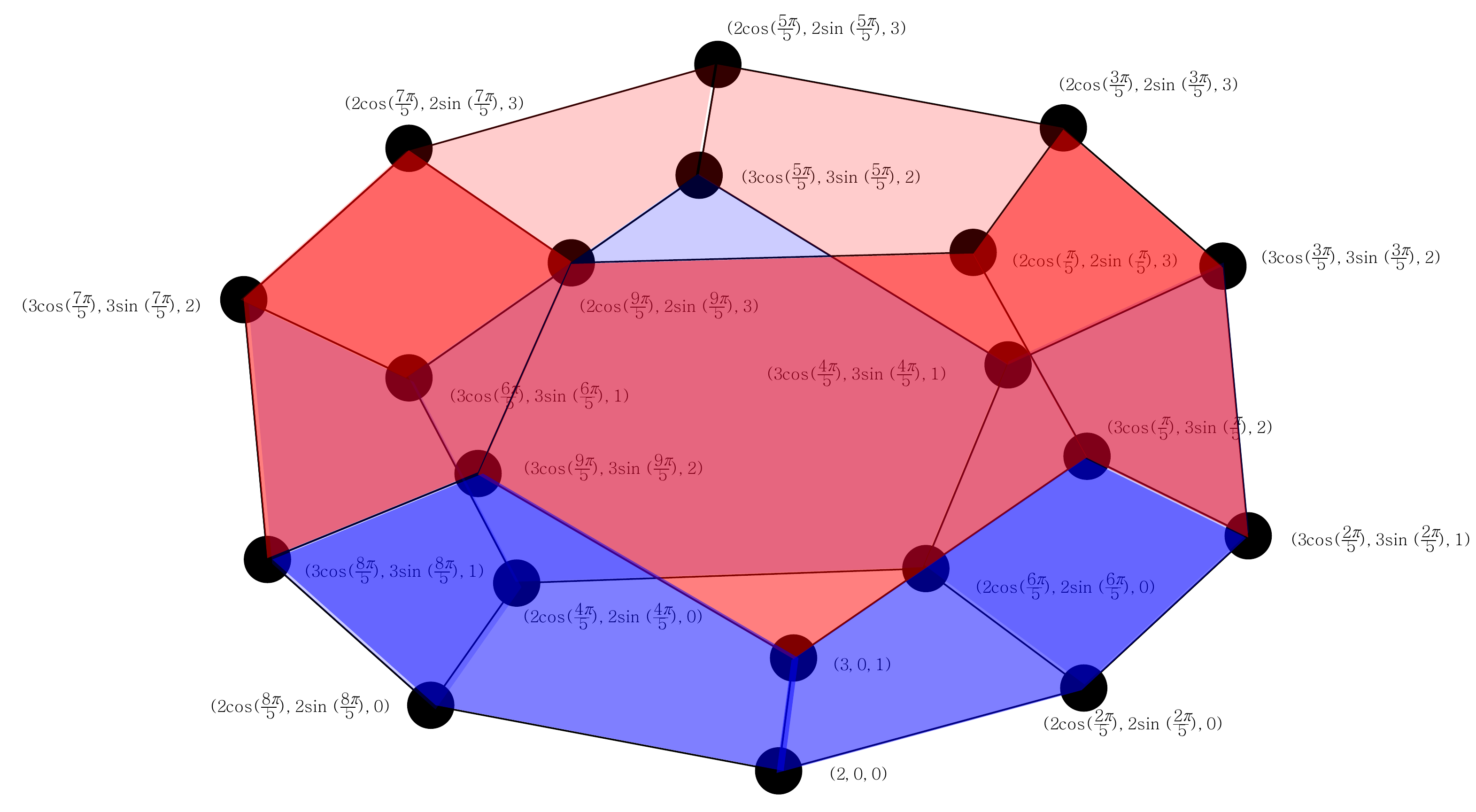}
\caption{Coordinates of vertices of the dodecahedron }
\label{dodecahedron}
\end{figure}

 \begin{figure}
\centering\includegraphics[width=300pt]{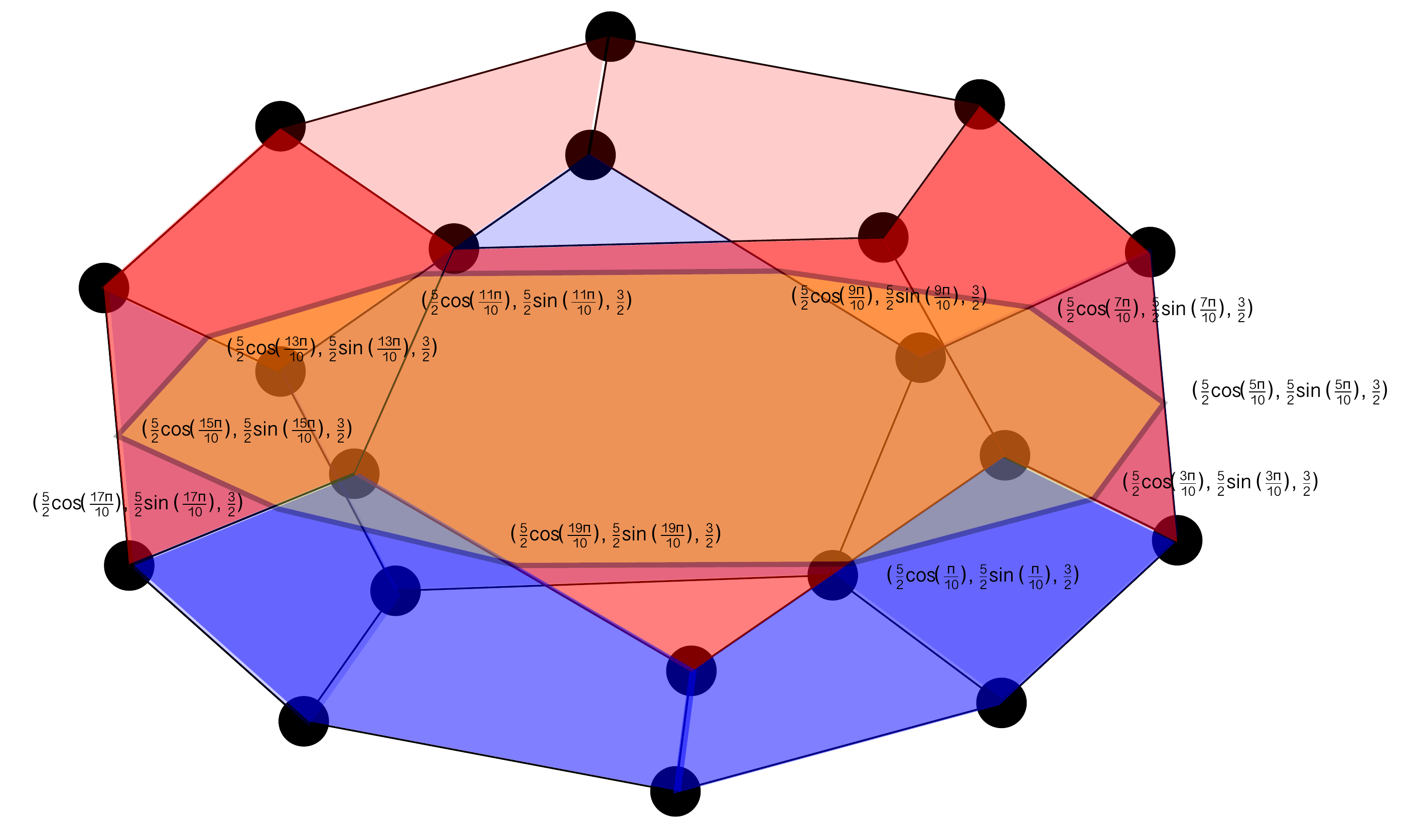}
\caption{Decahedron and decagon}
\label{dodecahedron-1}
\end{figure}

\section{A chain of decahedra. Self-interlocking structures} \label{Structure}
In order to construct a self-interlocking structure of decahedra, we need:

\begin{enumerate}
  \item  to block the faces of each decahedron inward and outward;
  \item  to block the decahedra when they rotate.
\end{enumerate}

The first goal can be achieved simply: we choose a $\Z_{5}$-symmetric decahedron and inside it we place the decagon of the largest possible area (this decagon will be parallel to the empty faces of the decahedron, see Fig.~\ref{dodecahedron-1}). But, this structure cannot block ``rotations'' of faces.
\begin{rem} 
Note that the structure of the decahedron can be changed by making the ``lower belt'' more ``gradual'' (so that we make the angle between the edges of the lower chord and the $Oxy$ plane smaller), and by making the ``upper belt'' more ``steep'' (so that the angle with the Oxy face is closer to line)
\end{rem}

Now let us construct an infinite chain $\mathcal{C}_{\infty}$ of nested decahedra (with faces with wings).

Namely, the $k$-th decahedron, $1 \leq k $ of the decahedron chain  $\mathcal {C}_{\infty}$ has coordinates:

$$\left\{\left(2\cos\left(\frac{(r(k)+2s\pi}{5}\right),2\sin\left(\frac{(r(k)+2s)\pi}{5}\right), 3(k-1)+t) \right)\right\}_{s=0,t=0}^{4\ \ \ 3}$$
where $r(k)$ is equal to $k$ modulo $2$. Note that the coordinates $x, y$ depend on the parity of $k$. The chain of decahedra $\mathcal{C}_{\infty}$ with wings is depicted in Fig.~\ref{wing_two_decahedron}.
 \begin{figure}
\centering\includegraphics[width=200pt]{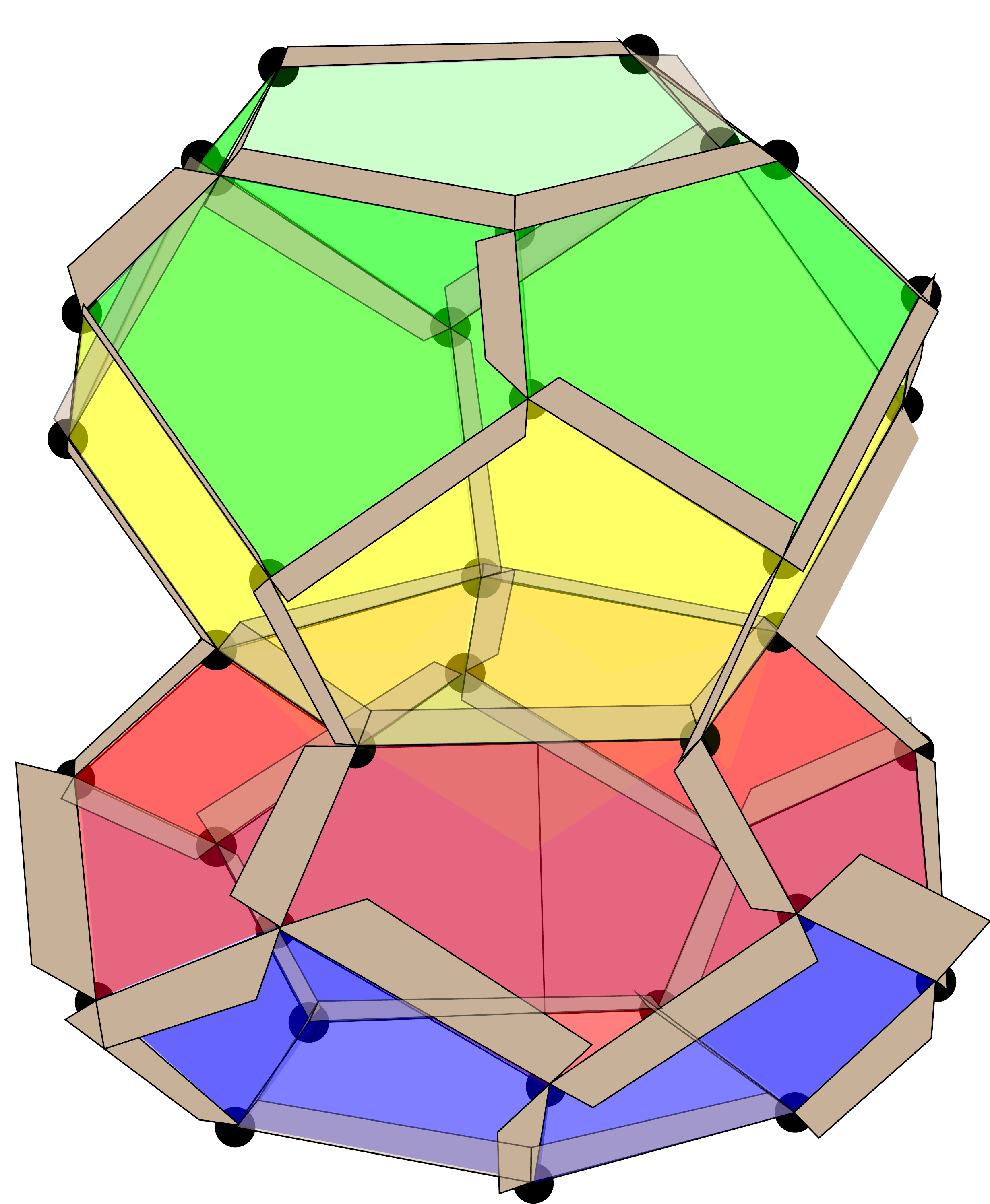}
\caption{$n$-th and $n+1$-th decahedra with faces with wings}
\label{wing_two_decahedron}
\end{figure}

\begin{lm}\label{lem:into}
Each face with wings of a decahedron of the chain $\mathcal{C}_{\infty}$ cannot move inward if other faces are fixed.
 \end{lm}

\begin{proof}[Proof]
Consider a face with wings of the lower belt (a blue face) of one decahedron in Fig.~\ref{into_face_R3}, say, the $n$-th decahedron. One can see that if the blue face moves in the direction of the vector $\overrightarrow{v}$, then it is blocked by the faces $A$, $C$ and the wing $B$. Note that $B$ is the wing of the upper belt of the $(n-1)$-th decahedron. From this observation it follows that the face of the lower belt is blocked in the direction of the vector directed outward. Similarly, it can be shown that a face with wings of the upper belt (red face) of the decahedron is blocked by two adjacent faces $D$ and $E$, and the proof is complete.

 \begin{figure}
\centering\includegraphics[width=200pt]{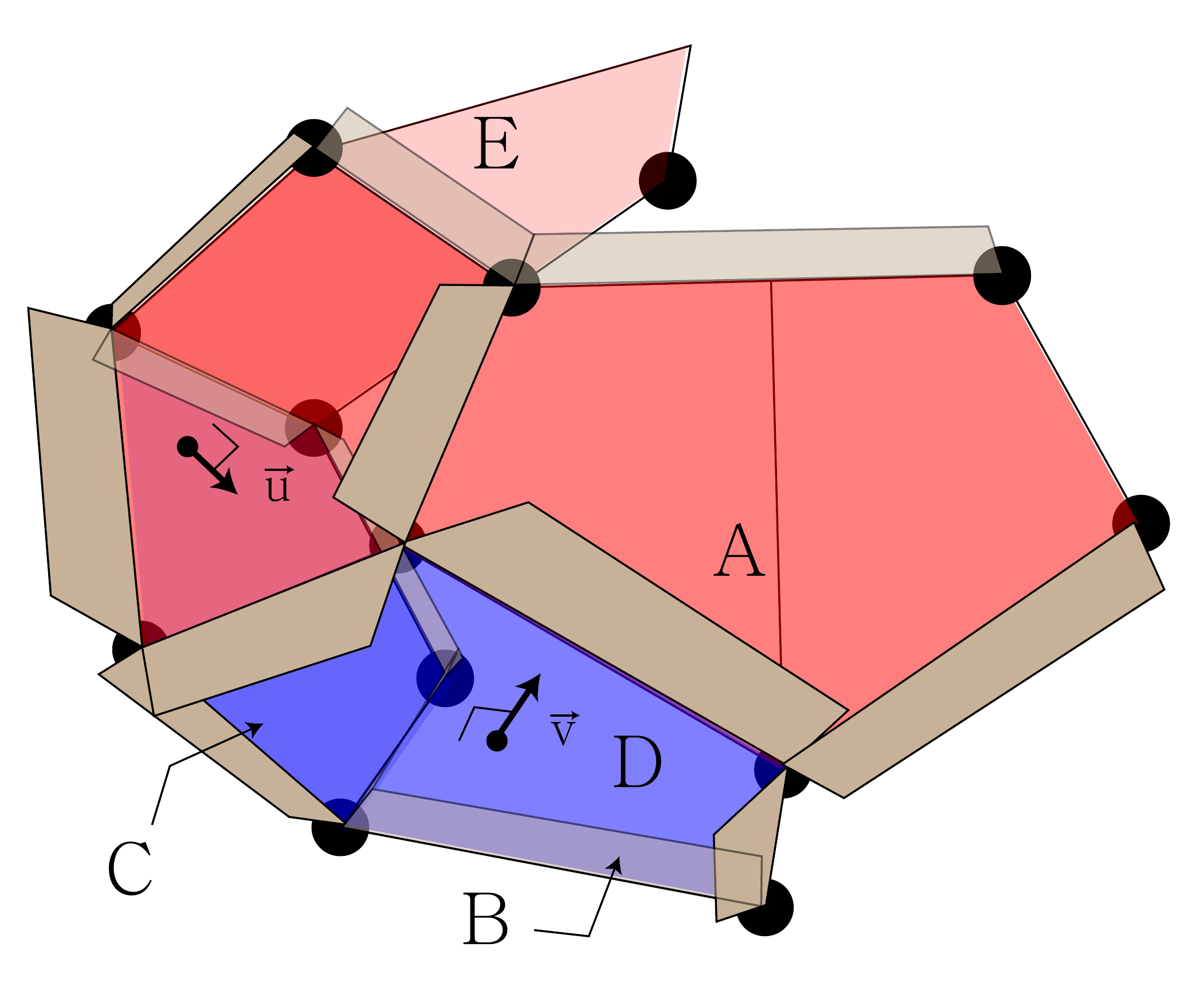}
\caption{Vectors directed inward}
\label{into_face_R3}
\end{figure}

\end{proof}

\begin{lm}\label{lem:slide}
A face with wings of a decahedron of the chain $\mathcal{C}_{\infty}$ cannot move along any vectors on the face if other faces are fixed.
 \end{lm}

\begin{proof}[Proof]
Consider the face of the upper belt of the lower decahedron (red face) in Fig. \ref{slide_decahedron}.
 \begin{figure}
\centering\includegraphics[width=200pt]{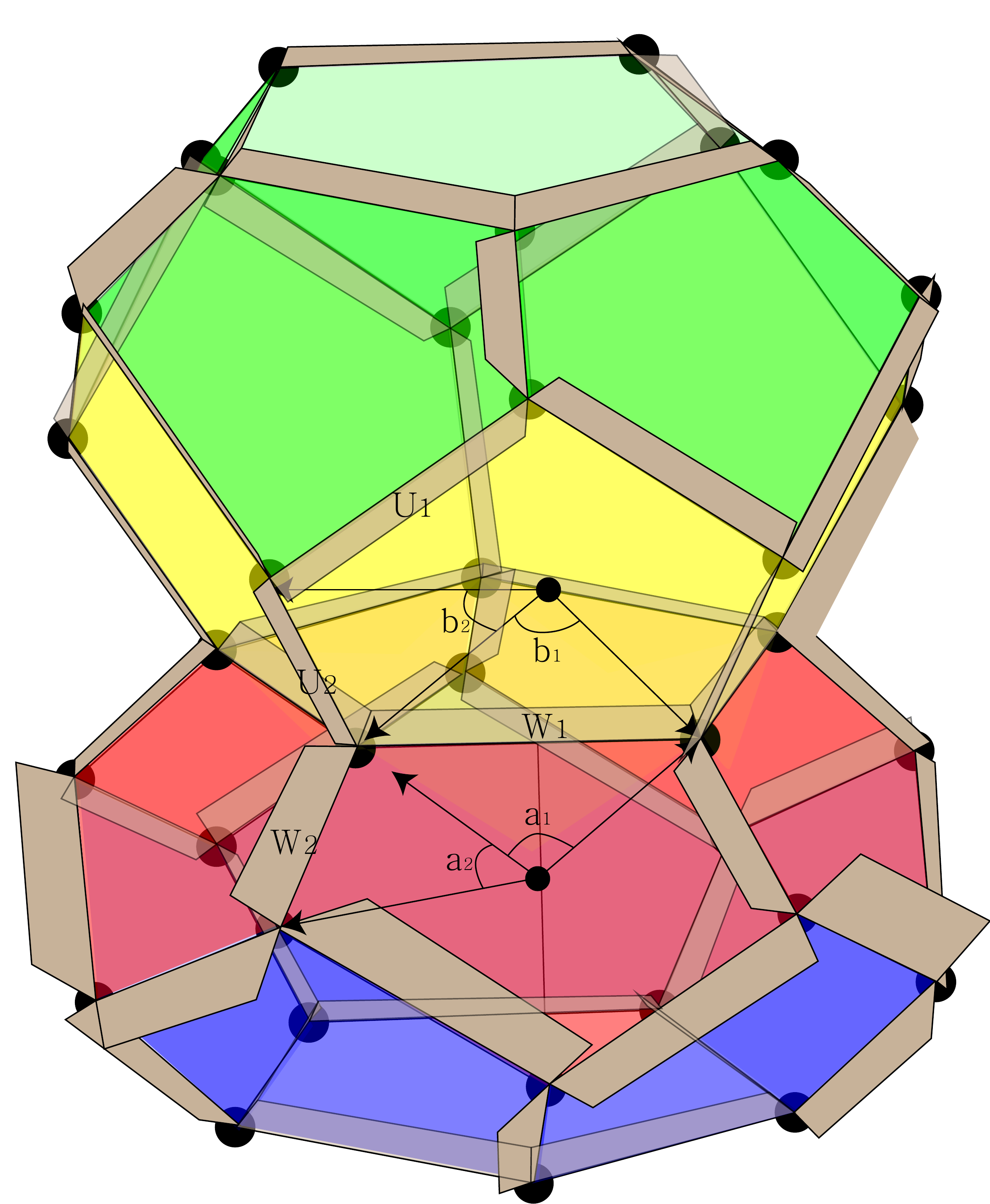}
\caption{Slides of faces in the direction of the vector on the face with the starting point $O_{1}$ and $O_{2}$}
\label{slide_decahedron}
\end{figure}
Let $\overrightarrow{v}_{1}$, $\overrightarrow{v}_{2}$ and $\overrightarrow{v}_{3}$ be vectors on the face with the same starting point $O_{1}$. Let $a_{1},a_{2},a_{3}$ be angles between $\overrightarrow{v}_{1}$ and $\overrightarrow{v}_{2}$, $\overrightarrow{v}_{2}$ and $\overrightarrow{v}_{3}$, and $\overrightarrow{v}_{1}$ and $\overrightarrow{v}_{3}$, respectively.
When the face of the upper belt of the lower decahedron moves in the direction of a vector on the face placed in the angle $a_{1}$ with the starting point $O_{1}$, it is blocked by the faces of the lower belt of the upper decahedron (yellow faces) because of the wing $W_ {2}$. When it moves in the direction of the vector placed in the angle $a_{2}$ with the starting point $O_{1}$, it is blocked by the faces of the lower belt of the upper decahedron because of the wing $W_ {1}$. When the red face moves in the direction of a vector on the face placed in the angle $a_{3}$ with the starting point $O_{1}$, it is easy to see that it is blocked.

Now, consider the face of the lower belt of the upper decahedron (yellow face) in Fig. \ref{slide_decahedron}. 
Let $\overrightarrow{u}_{1}$, $\overrightarrow{u}_{2}$ and $\overrightarrow{u}_{3}$ be vectors on the face with the same starting point $O_{2}$. Let $b_{1},b_{2},b_{3}$ be angles between $\overrightarrow{u}_{1}$ and $\overrightarrow{u}_{2}$, $\overrightarrow{u}_{2}$ and $\overrightarrow{u}_{3}$, and $\overrightarrow{u}_{1}$ and $\overrightarrow{u}_{3}$, respectively. When the yellow face moves in the direction of the vector on the face placed in the angle $b_{1}$ with the starting point $O_{2}$, it is blocked by the red edges because of the wing $U_{2}$. When the yellow face moves in the direction of the vector on the face placed in the angle $b_{2}$ with the starting point $O_{2}$ it is blocked by the wing $U_{1}$. When the yellow face moves in the direction of a vector on the face placed in the angle $b_{3}$ with the starting point $O_{2}$ it is easy to see that it is blocked.

\end{proof}

\begin{lm}\label{lem:rotation}
Each face with wings of a decahedron of the chain $\mathcal{C}_{\infty}$ cannot rotate if other faces are fixed.
 \end{lm}

\begin{proof}[Proof]

It is clear that each face of the chain $\mathcal{C}_{\infty}$ of decahedra cannot rotate around a vector perpendicular to it.

Now let us show that the face of the chain $\mathcal{C}_{\infty}$ of decahedra cannot rotate around its sides. Consider a face of the lower belt (a blue face) of one decahedron in Fig.~\ref{rotate_under_faces_R3}, say, the $n$-th decahedron.

 \begin{figure}
\centering\includegraphics[width=250pt]{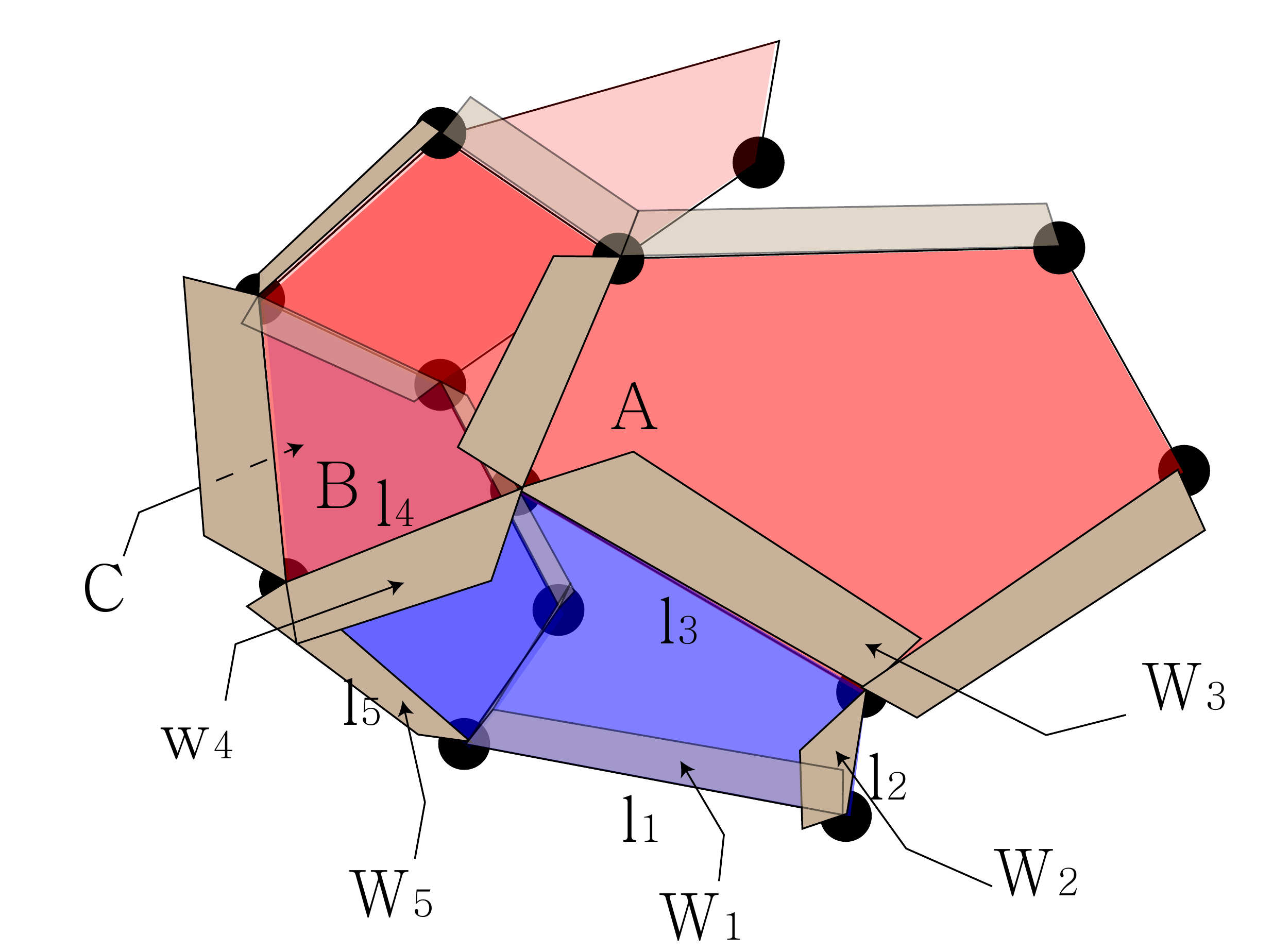}
\caption{Rotation of faces of the lower belt}
\label{rotate_under_faces_R3}
\end{figure}

When rotating the blue face around the side $l_ {1}$ outward, it is blocked by the wings $W_ {2}$ and $W_{4}$, and, when rotating around $l_{1}$ inward, it is blocked by the face $A$ because of the wing $W_ {3}$. When rotating the blue face around $l_{2}$ outward, it is blocked by the wing $W_{4}$, and when rotating around $l_{1}$ inside it is blocked by the face $A$ because of the wings $W_{3}$ and $W_{3}$ and by the face $C$ because of $W_{5}$. Similarly, it can be shown that the side edge of the lower chord is blocked when rotating around $l_{3}, l_{4}$ and $ l_{5} $.

Consider a face of the upper belt (a red face) of one decahedron in Fig.~\ref{rotate_over_faces_R3}.
 \begin{figure}
\centering\includegraphics[width=250pt]{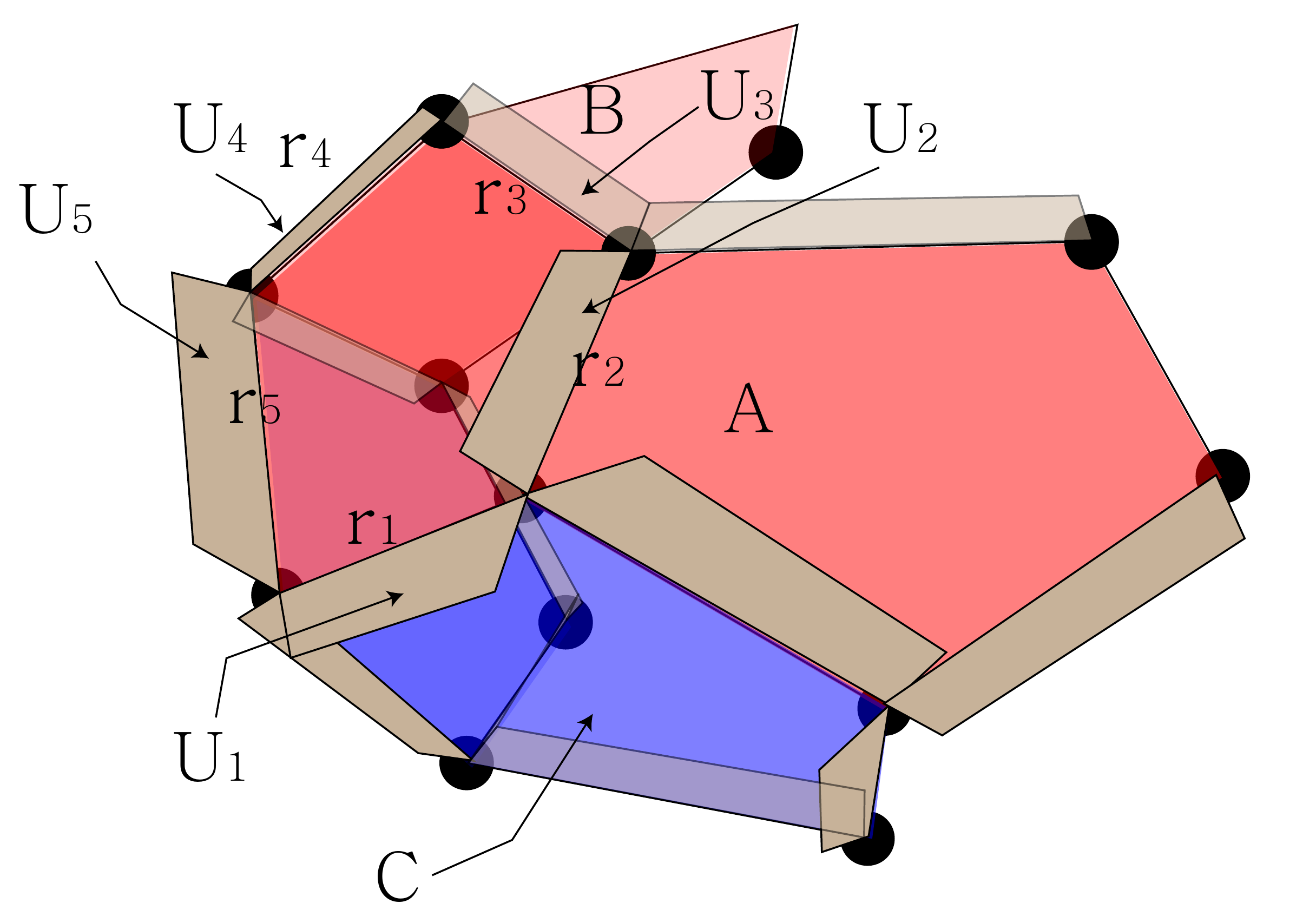}
\caption{Rotation of faces of the upper belt}
\label{rotate_over_faces_R3}
\end{figure}
When the red face rotates outward around $r_{1}$, it is blocked by the wings $U_{2}$ and $U_{5} $, and, when it is rotated around $r_{1}$ inward, it is blocked by the face $B$ due to for the $ U_ {4} $ wings. When the red face rotates outward around $r_{2}$, it is blocked by the wings $U_ {5}$ and when it is rotated around $r_{1}$ inward, then it is blocked by $ B $ because of $U_{4}$ and $C$ because of $U_{1}$. Similarly, it can be shown that the side edge of the lower chord is blocked when rotating around $ r_ {3}, r_ {4} $ and $ r_ {5} $.
\end{proof}

The lemmas \ref{lem:out}, \ref{lem:into}, \ref{lem:slide} and \ref{lem:rotation} imply the following theorem:

\begin{thm}\label{thm:main}
The faces of the chain $ \mathcal{C}_{\infty}$ of infinite decahedra cannot move.
\end{thm}

\section{Necklace of decahedra. Self-interlocking structure}             \label{Structure1}

{\em The necklace $\mathcal{M}_{2n}$ of $2n$ decahedra} can be constructed in the following way: Consider pentagons on the plane $Oxy$ with the following vertices $P_1, P_2, P_3, P_4$:

$$P_{1} = \left\{\left(2\cos\left(\frac{2s\pi}{5}\right),2\sin\left(\frac{2s\pi}{5}\right), 0\right)\right\}_{s=0}^4,$$

$$P_{2} = \left\{\left(3\cos\left(\frac{2s\pi}{5}\right),3\sin\left(\frac{2s\pi}{5}\right), 0\right)\right\}_{s=0}^4,$$

 $$P_{3} = \left\{\left(3\cos\left(\frac{(2s+1)\pi}{5}\right),3\sin\left(\frac{(2s+1)\pi}{5}\right), 0\right)\right\}_{s=0}^4$$

  $$\mbox{and}\quad P_{4} = \left\{\left(2\cos\left(\frac{(2s+1)\pi}{5}\right),2\sin\left(\frac{(2s+1)\pi}{5}\right), 0\right)\right\}_{s=0}^4.$$

   Note that the vertices of the $k$-th decahedron are obtained by the parallel translation of the pentagons $P_{1} $, $P_{2}$, $P_{3} $ and $P_{4} $ relative to the vector $(0,0,1)$ on planes $ z = 3 (k-1) $, $ z = 3 (k-1) + 1 $, $ z = 3 (k-1) + 2 $ and $ z = 3 (k-1) + 3$, respectively.

As above, the $ k $-th decahedron of the necklace $\mathcal{M}_{2n} $ is constructed by rotating the pentagons $P_{1}$, $P_{2}$, $P_{3}$ and $P_{4}$ to $Oxy$ on angles $ (3 (k-1) + l) \cdot \frac{2 \pi}{3n} $, $ l = 0,1,2,3 $, respectively. In other words, first, we transfer the pentagons $P_{1}$, $ P_{2} $, $ P_{3} $ and $P_{4}$ parallelly along the $Oz$ axis on the $Oxy$ plane and denote them by $P'_{1,0} $, $ P'_{2,0} $, $ P'_{3,0} $ and $ P'_{4,0} $. Rotate the plane $ Oxy $ by an angle $ \theta $ and denote the image by $ O_{\theta} $. The image of the vertices $ P'_{i, 0} $ when the plane $ Oxy$ is rotated through the angle $ \theta $ is denoted by $ P'_{i, \theta} $, $ i = 1,2 $.

The obtained vertices
$$P'_{1,0}, P'_{2,\frac{2\pi}{3n}}, P'_{3, \frac{4\pi}{3n}}, P'_{4, \frac{6\pi}{3n}}$$ 
become the vertices of the first decahedron, see. Fig.~\ref{cont_1_decahedron}. Vertices of the second decahedron are
$$P'_{4, \frac{6\pi}{3n}}, P'_{3, \frac{8\pi}{3n}}, P'_{2, \frac{10\pi}{3n}}, P'_{1, \frac{12\pi}{3n}},$$
and
$$\quad P'_{1, \frac{12\pi}{3n}}, P'_{2, \frac{14\pi}{3n}}, P'_{3, \frac{16\pi}{3n}},P'_{4, \frac{18\pi}{3n}}$$ are vertices of $3$rd decahedron, and this process can be repeated as many as one wants.

 \begin{figure}
\centering\includegraphics[width=200pt]{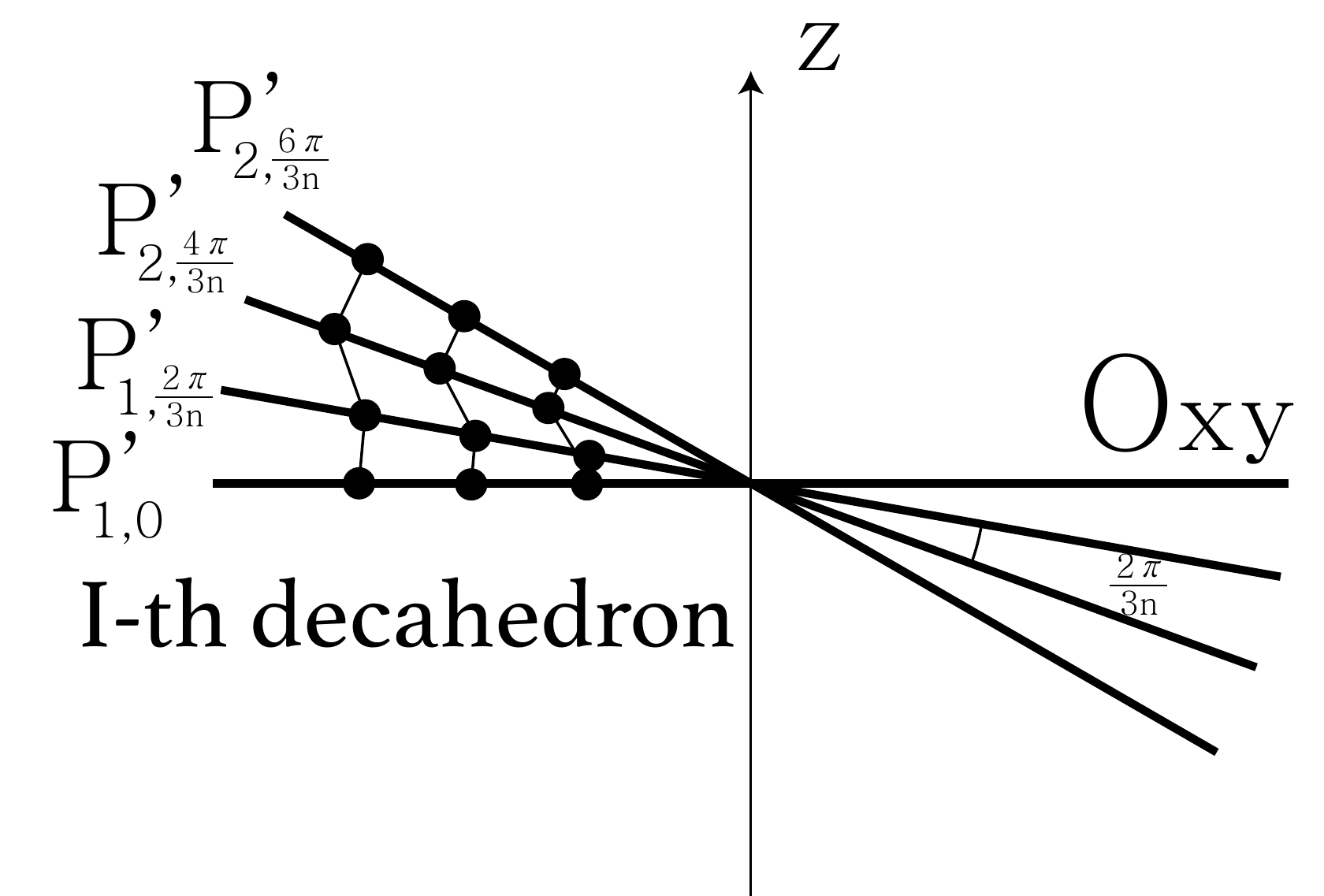}
\caption{First decahedron of the neckless of $n$ decahedra}
\label{cont_1_decahedron}
\end{figure}



The necklace $\mathcal{M}_{2n}$ of $2n$ decahedra consists of $2n$ decahedra ($20n$ faces) embedded in the space $\mathbb{R}^{3}$ such that $k$-th decahedron is placed with the following coordinates: \\


\begin{eqnarray*}
\begin{pmatrix}
1&0 & 0 \\
0 & \cos\left((3(k-1))\frac{2\pi}{3n}\right) & -\sin\left((3(k-1))\frac{2\pi}{3n}\right) \\
0& \sin\left((3(k-1))\frac{2\pi}{3n}\right) & \cos\left((3(k-1))\frac{2\pi}{3n}\right)
\end{pmatrix}\times \\ {}\\
\times \left(2\cos\left(\frac{(r(k)+2s)\pi}{5}\right),2\sin\left(\frac{(r(k)+2s)\pi}{5}\right), 0\right),
\end{eqnarray*}

\begin{eqnarray*}
\begin{pmatrix}
1&0 & 0 \\
0 & \cos\left((3(k-1)+1)\frac{2\pi}{3n}\right) & -\sin\left((3(k-1)+1)\frac{2\pi}{3n}\right) \\
0& \sin\left((3(k-1)+1)\frac{2\pi}{3n}\right) & \cos\left((3(k-1)+1)\frac{2\pi}{3n}\right)
\end{pmatrix}\times \\ 
\times \left(3\cos\left(\frac{(r(k)+2s)\pi}{5} \right), 3\sin\left(\frac{(r(k)+2s)\pi}{5} \right), 0 \right),\\
\begin{pmatrix}
1&0 & 0 \\
0 & \cos\left((3(k-1)+2)\frac{2\pi}{3n}\right) & -\sin\left((3(k-1)+2)\frac{2\pi}{3n}\right) \\
0& \sin\left((3(k-1)+2)\frac{2\pi}{3n}\right) & \cos\left((3(k-1)+2)\frac{2\pi}{3n}\right)
\end{pmatrix}\times \\ 
\times \left(3\cos\left(\frac{(r(k)+2s+1)\pi}{5}\right),3\sin\left(\frac{(r(k)+2s+1)\pi}{5}\right), 0\right), \quad\mbox{and}
\end{eqnarray*}

\begin{eqnarray*}
\begin{pmatrix}
1&0 & 0 \\
0 & \cos\left((3(k-1)+3)\frac{2\pi}{3n}\right) & -\sin\left((3(k-1)+3)\frac{2\pi}{3n}\right) \\
0& \sin\left((3(k-1)+3)\frac{2\pi}{3n}\right) & \cos\left((3(k-1)+3)\frac{2\pi}{3n}\right)
\end{pmatrix}\times \\ {}\\
\times \left(2\cos\left(\frac{(r(k)+2s+1)\pi}{5}\right),2\sin\left(\frac{(r(k)+2s+1)\pi}{5}\right), 0\right),
\end{eqnarray*}
where $s=0,\dots,4$, a $r(k)$ is equal to $k$ modulo $2$.
Analogously to Theorem~\ref{thm:main} the following theorem can be proved.
\begin{thm}
Faces of the necklace $\mathcal{M}_{2n}$ of $2n$ decahedra cannot be moved.
\end{thm}



\section{Questions and stories for further research}     \label{Openproblems}
The phenomenon of {\em self-interlocking} is important not only by itself, but also as a direction to the engineering applications of combinatorial geometry, which can be quite varied and unexpected. The fact that self-interlocking structures were invented only at the turn of the century says, on the one hand, of a lack of understanding of three-dimensional space, and on the other hand, of the possibilities that this (future) understanding provides. This raises a number of questions. Some of them are related to engineering or architectural applications, and some are related to the development of intuition and understanding of the general situation.



Let us start with the following questions. 
\begin{enumerate}
\item How to realise various well-known structures (Walls corresponding to action of Coxeter groups, Voronoi tilings, ...) by using self-interlocking structure, that is, how to build the walls correctly, so as not to glue them together, but to jam?

\item How to construct {\em flexible} self-interlocking structures, i.e., self-interlocking structures which admit infinitesimal moves but the whole space of positions (with some faces fixed) is bounded.

\item  How to construct self-interlocking structure in dimension 4,\\

(a) what is known about 4-dimension?\\
(b) how to construct 3 dimensional self-interlocking structure\\
(c) how to construct 2 dimensional self-interlocking structure\\

\item  Are there any connections here with the packages of something?
Probably 3-dimensional space has it. (you can try to pack everything very tightly),there is probably no lower bound: two-dimensional space gives volume 0.

\item Is it possible to arrange a self-interlocking structure in a torus or in a cylinder that would "hold on by itself" - without holding two polygons?
There are no homotopies in a specific torus with fixed lengths.

\end{enumerate}

\section{Acknowledgement}
The authors are extremely grateful to Fedoseev, who paid attention to the problem of self-interlocking and repeatedly discussed intermediate results before their publication.\\
Ideas of constructions in this paper are contributed by the first author and third author found exact coordinates for realisation of given structure in three dimensional space.\\
Alexei Kanel-Belov was supported by Russian Science Foundation  grant No. 17-11-01377.\\
The work of V.O.Manturov was funded by the development program of the Regional Scientific and Educational Mathematical Center of the Volga Federal District, agreement N 075-02-2020.

\renewcommand{\refname}{References}

\pagebreak

\section{Coordinates of new model with decahedra}

\begin{eqnarray*}
\begin{pmatrix}
1&0 & 0 \\
0 & \cos\left((3(k-1))\frac{2\pi}{3n}\right) & -\sin\left((3(k-1))\frac{2\pi}{3n}\right) \\
0& \sin\left((3(k-1))\frac{2\pi}{3n}\right) & \cos\left((3(k-1))\frac{2\pi}{3n}\right)
\end{pmatrix}\times \\
\times \left(2\cos\left(\frac{(2s)\pi}{5}\right),2\sin\left(\frac{(2s)\pi}{5} \right), 0 \right),
\end{eqnarray*}

\begin{eqnarray*}
\begin{pmatrix}
1&0 & 0 \\
0 & \cos\left((3(k-1)+1)\frac{2\pi}{3n}\right) & -\sin\left((3(k-1)+1)\frac{2\pi}{3n}\right) \\
0& \sin\left((3(k-1)+1)\frac{2\pi}{3n}\right) & \cos\left((3(k-1)+1)\frac{2\pi}{3n}\right)
\end{pmatrix}\times \\
\times \left(3\cos\left(\frac{(2s)\pi}{5}\right),3\sin\left(\frac{(2s)\pi}{5}\right), 0\right),
\end{eqnarray*}

\begin{eqnarray*}
\begin{pmatrix}
1&0 & 0 \\
0 & \cos\left((3(k-1)+2)\frac{2\pi}{3n}\right) & -\sin\left((3(k-1)+2)\frac{2\pi}{3n}\right) \\
0& \sin\left((3(k-1)+2)\frac{2\pi}{3n}\right) & \cos\left((3(k-1)+2)\frac{2\pi}{3n}\right)
\end{pmatrix}\times \\
\times \left( 3\cos\left(\frac{(2s+1)\pi}{5}\right),3\sin\left(\frac{(2s+1)\pi}{5}\right), 0\right),
\end{eqnarray*}
and
\begin{eqnarray*}
\begin{pmatrix}
1&0 & 0 \\
0 & \cos\left((3(k-1)+3)\frac{2\pi}{3n}\right) & -\sin\left((3(k-1)+3)\frac{2\pi}{3n}\right) \\
0& \sin\left((3(k-1)+3)\frac{2\pi}{3n}\right) & \cos\left((3(k-1)+3)\frac{2\pi}{3n}\right)
\end{pmatrix}\times \\
\times \left(2\cos\left(\frac{(2s+1)\pi}{5}\right),2\sin\left(\frac{(2s+1)\pi}{5}\right), 0\right),
\end{eqnarray*}

where $s=0,\dots, 4$

 \begin{figure}
\centering\includegraphics[width=300pt]{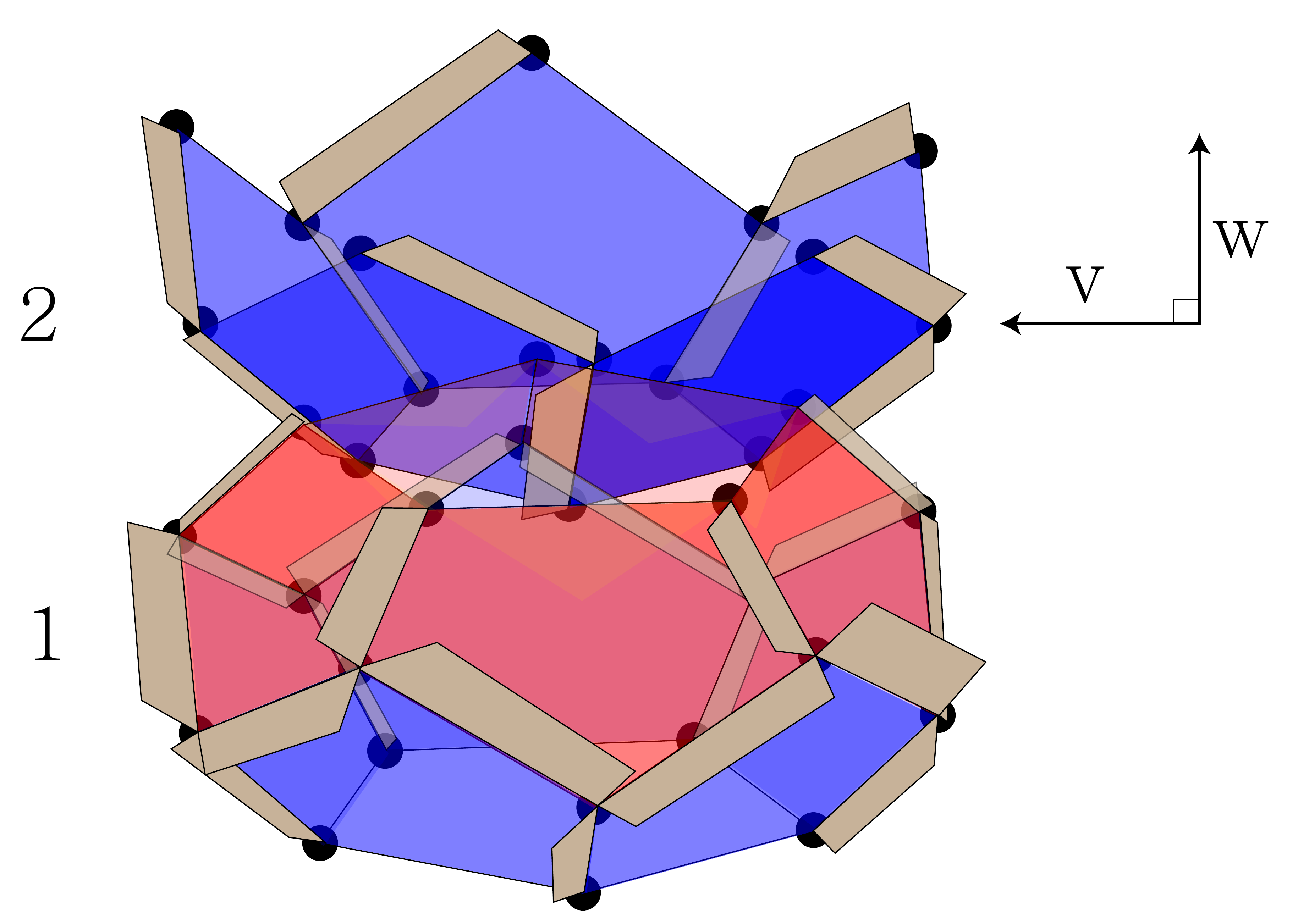}
\caption{New model and two decahedra of neckless of $n$ decahedra}
\label{new_idea_wing_decahedronpdf}
\end{figure}

\end{document}